\documentclass[a4paper,12pt]{article}
\usepackage{amsmath,amsthm,amsfonts,amssymb,bbm}
\usepackage{graphicx,psfrag,subfigure}
\usepackage{cite}
\usepackage{hyperref}
\usepackage[hmargin=1in,vmargin=1in]{geometry}
\hypersetup{colorlinks=true}

\renewcommand{\d}{\mathrm d}

\newcommand{\R}{\mathbb R}
\newcommand{\Z}{\mathbb Z}
\newcommand{\wt}{\widetilde}
\newcommand{\wh}{\widehat}
\renewcommand{\Re}{\operatorname{Re}}
\renewcommand{\Im}{\operatorname{Im}}
\newcommand{\Ai}{\operatorname{Ai}}

\newcommand{\C}{\mathcal{C}}
\newcommand{\D}{\mathcal{D}}

\newcommand{\id}{\mathbbm{1}}
\renewcommand{\O}{\mathcal{O}}
\renewcommand{\P}{\mathbf P}
\newcommand{\E}{\mathbf E}

\newcommand{\I}{\mathrm i}

\DeclareMathOperator*{\sgn}{sgn}

\newtheorem{proposition}{Proposition}[section]
\newtheorem{theorem}[proposition]{Theorem}
\newtheorem{lemma}[proposition]{Lemma}
\newtheorem{definition}[proposition]{Definition}
\newtheorem{corollary}[proposition]{Corollary}

\theoremstyle{definition}
\newtheorem{remark}[proposition]{Remark}

\newtheorem{conj}[proposition]{Conjecture}

\numberwithin{equation}{section}

\author{Patrik L.\ Ferrari\thanks{Institute for Applied Mathematics, Bonn University, Endenicher Allee 60, 53115 Bonn,
Germany. E-mail: {\tt ferrari@uni-bonn.de}} \and
B\'alint Vet\H o\thanks{Institute for Applied Mathematics, Bonn University, Endenicher Allee 60, 53115 Bonn, Germany;
MTA--BME Stochastics Research Group, Egry J.\ u.\ 1, 1111 Budapest, Hungary. E-mail: {\tt vetob@math.bme.hu}}}

\title{Tracy--Widom asymptotics for \mbox{$q$-TASEP}}
\date{}

\begin{document}
\maketitle
\sloppy

\begin{abstract}
We consider the \mbox{$q$-TASEP} that is a $q$-deformation of the totally asymmetric simple exclusion process (TASEP) on $\Z$ for $q\in[0,1)$ where the jump rates depend on the gap to the next particle.
For step initial condition, we prove that the current fluctuation of \mbox{$q$-TASEP} at time $\tau$ is of order $\tau^{1/3}$ and asymptotically distributed as the GUE Tracy--Widom distribution, which confirms the KPZ scaling theory conjecture.


\noindent {\bf MSC classes:} 60K35, 60B20
\end{abstract}

\section{Introduction}

The totally asymmetric simple exclusion process (TASEP) is the simplest non-reversible stochastic interacting particle system on $\Z$.
Particles try to jump to the right by $1$ according to independent Poisson processes with rate $1$ and the jumps are allowed only if the target site is empty (due to the exclusion constraint).
There has been a lot of studies around this model (and its discrete time versions).
For instance, the limiting process for particles positions or for the current fluctuations are given by the Airy processes.
This was obtained using determinantal structures of correlation functions~\cite{Jo03b,BFPS06,BF07,Sas05}.

The \mbox{$q$-TASEP} is a generalization of TASEP defined as follows.
For a parameter $q\in[0,1)$, the jumps of the particles on $\Z$ are still independent of each other, but the jump rate is $1-q^{\rm gap}$
where the gap is the number of consecutive vacant sites next to the particle on its right. In the $q=0$ case, \mbox{$q$-TASEP} reduces to TASEP.

As a natural generalization of TASEP, the \mbox{$q$-TASEP} also belongs to the Kardar--Parisi--Zhang (KPZ) universality class,
hence by the universality conjecture it is expected to show the characteristic asymptotic fluctuation statistics of the KPZ class.
Indeed, in Theorem~\ref{thm:main} we prove that the large time current fluctuations are governed by the (GUE) Tracy--Widom distribution.
This confirms the KPZ scaling theory conjecture, see also \cite{Spo13a}.

The \mbox{$q$-TASEP} was first introduced by Borodin and Corwin in~\cite{BC11}.
They investigated the $q$-Whittaker 2d growth model which is an interacting particle system in two space dimensions on the space of Gelfand--Tsetlin patterns.
The \mbox{$q$-TASEP} is a Markovian subsystem of this two-dimensional process.
The key idea to study the $q$-Whittaker 2d growth model was to exploit its connection to the $q$-Whittaker process.
This connection was first observed by O'Connell in~\cite{OCon09} for a special case which can be obtained from $q\in[0,1)$ in the $q\to1$ limit in a rather careful way as explained in~\cite{BC11}.
The Macdonald process is above the $q$-Whittaker process in the algebraic hierarchy,
hence from the study of Macdonald processes, explicit formulas could be derived for expectations of relevant observables of the \mbox{$q$-TASEP} for a special class of initial conditions.
In particular for step initial condition, a Fredholm determinant formula was given by Borodin, Corwin and Ferrari in~\cite{BCF12} for the $q$-Laplace transform of the particle position
which is the starting point of the asymptotic analysis performed in this paper and it is cited as Theorem~\ref{thm:starting} below.

The totally asymmetric zero range process with a special choice of rate function which is also known as \mbox{$q$-TAZRP} was first introduced by Sasamoto and Wadati in~\cite{SW98b}.
The duality of \mbox{$q$-TASEP} and \mbox{$q$-TAZRP} was proved in~\cite{BCS12} and as a consequence,
joint moment formulas for multiple particle positions were obtained for \mbox{$q$-TASEP} however they are not of Fredholm determinant form.
An explicit formula for the transition probabilities of \mbox{$q$-TAZRP} with a finite number of particles was recently found in two different ways:
in~\cite{BCPS13} by using the spectral theory for the $q$-Boson particle system and in~\cite{KL13} by using the Bethe ansatz.
In these two papers, the distribution of the left-most particle's position after a fixed time for general initial condition was also characterized.
In~\cite{BC13}, two natural discrete time versions of \mbox{$q$-TASEP} were introduced and Fredholm determinant expressions were proved for the $q$-Laplace transform of the particle positions.
A further extension of \mbox{$q$-TASEP} appears in~\cite{CP13}, the $q$-PushASEP which is yet another integrable particle system, a $q$-generalization of the PushASEP studied in~\cite{BF07}.
An explicit contour integral formula was derived for the joint moment of particle positions,
but it is not completely clear how to turn it into a Fredholm determinant expression even for a single position.

Finally, let us point out the technical issues of this paper which are new and were not present for in the asymptotic analysis of the semi-discrete directed polymer in~\cite{BCF12}.
One difficulty lies in the choice of the contour for the Fredholm determinant: it has to be a steep descent path for the asymptotic analysis but also be such that the extra residues coming from the sine inverse can be controlled.
The contours that we have chosen are shown on Figure~\ref{fig:contours}.
Further, the real part of the function that gives asymptotically the main exponential contribution is periodic in the vertical direction.
The contour for the Fredholm determinant stays within one period, but the contour in the integral representation of the kernel is vertical,
and the convergence of this integral comes from the extra sine factor in the denominator of the integrand.

The paper is organized as follows.
Section~\ref{s:results} contains the main result of the paper with explanation.
The behaviour of the hydrodynamic limit of \mbox{$q$-TASEP} stated in Proposition~\ref{prop:lln} is verified in Section~\ref{s:hydro}.
Theorem~\ref{thm:starting}, the starting formula of our investigations with the necessary notation is given in Section~\ref{s:finite}.
The main result of the paper follows from Theorem~\ref{thm:starting} in two steps:
we prove in Section~\ref{s:qLaplace} that under proper scaling the \mbox{$q$-Laplace} transform converges to the distribution function,
and in Section~\ref{s:asympt} we show that the Fredholm determinant that appears in Theorem~\ref{thm:starting} converges to that of the Airy kernel.
Section~\ref{s:scalingtheory} is devoted to show that the KPZ scaling theory conjecture is satisfied for \mbox{$q$-TASEP}.

\section{Model and main result}\label{s:results}

We start with the definition of the \mbox{$q$-TASEP} with \emph{step initial condition} and further notation.
Let $q\in(0,1)$.
The \mbox{$q$-TASEP} is a continuous time interacting particle system on $\Z$ that consists of the evolution of particles ${\bf X}(\tau)=(X_N(\tau):N\in\Z\mbox{ or }N\in\mathbb N)$ for $\tau\ge0$.
The particles are numbered from right to left.
Each particle $X_N(\tau)$ jumps to the right by $1$ independently of the others with rate $1-q^{X_{N-1}(\tau)-X_N(\tau)-1}$.
The infinitesimal generator of the process is given by
\begin{equation}
(Lf)({\bf x})=\sum_k\left(1-q^{x_{k-1}-x_k-1}\right)\left(f\left({\bf x}^k\right)-f({\bf x})\right)
\end{equation}
where ${\bf x}=(x_N:N\in\Z\mbox{ or }N\in\mathbb N)$ is a configuration of particles such that $x_N<x_{N-1}$ for all $N$ and ${\bf x}^k$ is the configuration where $x_k$ is increased by one.
Note that this only happens when the position $x_{k-1}$ is not at $x_k+1$ and that the dynamics preserves the order of particles.
Step initial condition means that the particles are initially filling the negative integers, i.e., there are only particles with labels $N=1,2,\dots$ and they are initially at $X_N(0)=-N$.

\begin{definition}\label{def:qdigamma}
Fix $q\in(0,1)$.
The infinite \emph{$q$-Pochhammer symbol} is given by
\begin{equation}\label{defqpochhammer}
(a;q)_\infty=\prod_{k=0}^\infty(1-aq^k)
\end{equation}
for any $a\in\mathbb C$.
Let
\begin{equation}\label{defqgamma}
\Gamma_q(z)=(1-q)^{1-z}\frac{(q;q)_\infty}{(q^z;q)_\infty}
\end{equation}
be the \emph{$q$-gamma function}.
Then the \emph{$q$-digamma function} is defined by
\begin{equation}\label{defqdigamma}
\Psi_q(z)=\frac\partial{\partial z}\log\Gamma_q(z).
\end{equation}
\end{definition}

\begin{definition}\label{def:kfa}
Let $q\in(0,1)$ be fixed and choose a parameter $\theta>0$.
To these values, we associate the parameters
\begin{align}
\kappa\equiv \kappa(q,\theta)&=\frac{\Psi'_q(\theta)}{(\log q)^2q^\theta},\label{defkappa}\\
f\equiv f(q,\theta)&=\frac{\Psi'_q(\theta)}{(\log q)^2}-\frac{\Psi_q(\theta)}{\log q}-\frac{\log(1-q)}{\log q},\label{deff}\\
\chi\equiv\chi(q,\theta)&=\frac{\Psi'_q(\theta)\log q-\Psi''_q(\theta)}2.\label{defalpha}
\end{align}
\end{definition}
The parameters $f$ and $\kappa$ describe the global behaviour of the particle system as explained below.
In turns out that explicit formulas are available for them only in terms of the parameter $\theta$ which appears naturally in the asymptotic analysis of the problem.
To keep the notation simple, we will not write the $(q,\theta)$ dependence of $\kappa$, $f$ and $\chi$ in the sequel.
For $\kappa$ and $f$, there are the following series representations
\begin{equation}
\kappa=\sum_{k=0}^\infty \frac{q^k}{(1-q^{\theta+k})^2},\qquad f=\sum_{k=0}^\infty \frac{q^{2\theta+2k}}{(1-q^{\theta+k})^2}.\label{kappafseries}
\end{equation}

The macroscopic picture of \mbox{$q$-TASEP} is as follows.
Due to particle conservation, under hydrodynamic limit, one expects that the (macroscopic) particle density $\rho(t,x)$ satisfies the PDE
\begin{equation}\label{pde}
\frac\partial{\partial t}\rho(t,x)+\frac\partial{\partial x}j(\rho)(t,x)=0
\end{equation}
with initial condition $\rho(0,x)=\id(x<0)$ where $j(\rho)$ is the particle current at density $\rho$.

It was shown in~\cite{BC11} that in the stationary distribution of the \mbox{$q$-TASEP} dynamics, gaps between consecutive particles are i.i.d.\ $q$-geometric random variables with some parameter $\alpha\in[0,1)$, i.e.\ with distribution
\begin{equation}\label{qgeodistr}
\P({\rm gap}=k)=(\alpha;q)_\infty\frac{\alpha^k}{(q;q)_k},\quad k=0,1,2,\ldots.
\end{equation}
The \emph{local stationarity assumption} is that the gaps between particles have locally \mbox{$q$-geometric} distribution with some parameter $\alpha$.
Under this assumption, the macroscopic behaviour of $q$-TASEP can be deduced from the PDE \eqref{pde} as follows.

\begin{proposition}\label{prop:lln}$ $
\begin{enumerate}
\item
Under the local stationarity assumption where the parameter of the $q$-geometric distribution is $\alpha$, the particle density $\rho$ and the particle current $j(\rho)$ are given by
\begin{equation}\label{rhojrho}
\rho=\frac{\log q}{\log q+\log(1-q)+\Psi_q(\log_q\alpha)},\quad j(\rho)=\alpha \rho.
\end{equation}

\item
The function
\begin{equation}\label{solutionPDE}
\rho\left(t,\frac{f-1}{\kappa}t\right)=\frac{\log q}{\log q+\log(1-q)+\Psi_q(\theta)}
\end{equation}
solves the PDE \eqref{pde} with the corresponding $j(\rho)$ given by \eqref{rhojrho} and with $\kappa$ and $f$ defined in \eqref{defkappa}--\eqref{deff} understood as functions of $\theta$.

\item
Fix a $\theta>0$ and assuming local stationarity.
Then (\ref{rhojrho}) with $\alpha=q^\theta$ gives the local particle density $\rho$ at time $\tau$ and position $(f(\theta)-1)\tau/\kappa(\theta)$ as well as the particle current $j(\rho)$ at position $(f(\theta)-1)\tau/\kappa(\theta)$.

\item
Consequently, under the assumption of local stationarity, the law of large numbers
\begin{equation}\label{lln}
\frac{X_N(\tau=\kappa N)}{N} \simeq f-1
\end{equation}
holds for the position of the $N$th particle after time $\kappa N$ as $N\to\infty$.
\end{enumerate}
\end{proposition}

In order to visualize the macroscopic behaviour predicted by \eqref{lln}, consider the evolution of the points $(X_N(\tau)+N,N)$ in the coordinate system.
For $\tau=0$, these points all lie on the positive half of the vertical axis.
For $\tau$ large and after rescaling the picture by $\tau$, the points are macroscopically around $(f/\kappa,1/\kappa)$
which is a curve that can be parameterized by $\theta$ and it is shown on Figure~\ref{fig:macro}.
It can easily be seen from the series representation of $\kappa$ and $f$ \eqref{kappafseries} that the curve touches the axes at $(1,0)$ for $\theta\to0$ and at $(0,1-q)$ for $\theta\to\infty$.
It is clear that the right-most \mbox{$q$-TASEP} particle has speed $1$ hence the touching point $(1,0)$,
whereas it follows from the above calculation that the left-most particle which has already started moving after time $\tau$ is around the position $-(1-q)\tau$.

\begin{figure}
\begin{center}
\includegraphics[width=200pt]{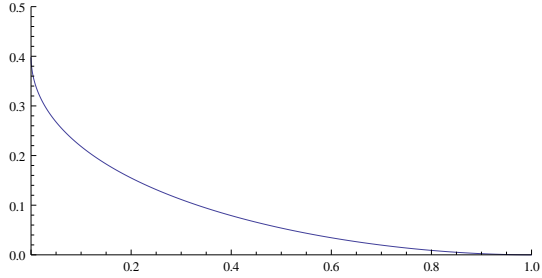}
\end{center}
\caption{The macroscopic shape of the positions of \mbox{$q$-TASEP} particles for $q=0.6$
which is the parametric curve $(f/\kappa,1/\kappa)$.
\label{fig:macro}}
\end{figure}

It is also possible to parameterize the macroscopic position by $\kappa$.
Indeed, from the series representation for $\kappa$ given in \eqref{kappafseries} one sees that $\theta\mapsto\kappa(\theta)$ is strictly decreasing from $\infty$ to $1/(1-q)$ as $\theta$ ranges from $0$ to $\infty$.
Hence the inverse function $\theta(\kappa)$ is well-defined however it is not explicit.
The present parametrization is also natural, because $\theta$ is the position of the double critical point of the function that gives the main contribution in the exponent of the correlation kernel as seen in Section~\ref{s:asympt}.

The main objective of the present paper is the study of the fluctuations of particle $X_N$ around the macroscopic (deterministic) behaviour given by \eqref{lln}.
By KPZ universality, one expects that these fluctuations are of order $\O(N^{1/3})$ and have Tracy--Widom statistics (see the review~\cite{Fer10b}).
Further, at a given time $\tau=\kappa N$, particles are correlated over a scale $\O(N^{2/3})$ with the Airy$_2$ process as limit process.
The same limit process is expected to arise~\cite{CFP10b} for the position of a tagged particle $X_N$ at times of order $N^{2/3}$ away from $\kappa N$ (as it was shown for TASEP in~\cite{SI07}).
More precisely, for any $c\in\R$ consider the scaling
\begin{align}
\tau(N,c)&=\kappa N+cq^{-\theta}N^{2/3},\label{deftau}\\
p(N,c)&=(f-1)N+cN^{2/3}-c^2\frac{(\log q)^3}{4\chi}N^{1/3}\label{defp}
\end{align}
with $\kappa$, $f$ and $\chi$ given in Definition~\ref{def:kfa}. $p(N,c)$ is the macroscopic approximation of the particle position, obtained as follows.
Let $\tilde\theta$ be such that $\tau(N,c)=\kappa(\tilde\theta)N$.
Then $p(N,c)=(f(\tilde\theta)-1)N+\O(1)$.
The rescaled tagged particle position given by
\begin{equation}\label{defxi}
c\mapsto \xi_N(c)=\frac{X_N(\tau(N,c))-p(N,c)}{\chi^{1/3}(\log q)^{-1}N^{1/3}}
\end{equation}
is expected to converge to the Airy$_2$ process.
Our main result is the convergence of one-point distribution of $\xi_N$ to the Tracy--Widom distribution function~\cite{TW94}.

\begin{theorem}\label{thm:main}
Let $q\in(0,1)$ and $\theta>0$ be fixed such that $q^\theta\le1/2$.
For any $c,x\in\R$ and with the notation above, the rescaled position $\xi_N$ converges in distribution, i.e.,
\begin{equation}
\lim_{N\to\infty}\P(\xi_N<x)=F_{\rm GUE}(x)
\end{equation}
where $F_{\rm GUE}$ is the GUE Tracy--Widom distribution function.
\end{theorem}

\begin{remark}
We believe that the condition $q^\theta\le1/2$ is technical and that Theorem~\ref{thm:main} holds for any $\theta>0$.
This technical restriction comes from the difficulty of controlling the simple poles arising from the sine inverse in representation of the kernel and at the same time to prove that one has a steep descent path.
It could be possible to obtain Tracy--Widom asymptotics for small values of $\theta$ using the small contour representation Theorem~3.2.11 of~\cite{BC11}
instead of the infinite contour representation Theorem~4.13 of~\cite{BCF12}.
However this other representation can not be used to analyze the whole $q^\theta>1/2$ case, so we do not pursue in this direction.
\end{remark}

\begin{remark}\label{rem:parametertau}
One can choose the time $\tau$ of the $q$-TASEP process to be the free parameter that tends to $\infty$.
It means that \eqref{defxi} can be rewritten as
\begin{equation}
X_{N(\tau,c)}(\tau)=\wt p(\tau,c)+\frac{\chi^{1/3}}{\kappa^{1/3}\log q}\xi_\tau\tau^{1/3}
\end{equation}
with
\begin{align}
N(\tau,c)=&\frac\tau\kappa-c\frac{\tau^{2/3}}{\kappa^{5/3}q^{\theta}}+\frac{2c^2}{3\kappa^{7/3}q^{2\theta}}\tau^{1/3},\label{defNtau}\\
\wt p(\tau,c)=&\frac{f-1}\kappa\tau+c\left(\frac1{\kappa^{2/3}}-\frac{f-1}{\kappa^{5/3}q^{\theta}}\right)\tau^{2/3}
-c^2\left(\frac{(\log q)^3}{4\chi\kappa^{1/3}}+\frac2{3\kappa^{4/3}q^{\theta}}+\frac{2(f-1)}{3\kappa^{7/3}q^{2\theta}}\right)\tau^{1/3}
\end{align}
where $\xi_\tau$ has also asymptotically GUE Tracy--Widom distribution.
In order to keep the notation simpler, we do not write integer part in \eqref{defNtau}.
\end{remark}

\begin{remark}\label{rem:current}
Theorem~\ref{thm:main} can be interpreted as a statement about the current fluctuation as follows.
Let $H(t,y)$ be the number of particles on the right of position $y$ at time $t$.
Then the event $\{X_N(t)\le y\}$ is clearly the same as $\{H(t,y)<N\}$.
Therefore the result of Theorem~\ref{thm:main} is equivalent with
\begin{equation}\label{currentrandomin}
\lim_{\tau\to\infty}\P\left(H\left(\tau,\wt p(\tau,c)+\frac{\chi^{1/3}}{\kappa^{1/3}\log q}x\tau^{1/3}\right)<N(\tau,c)\right)=F_{\rm GUE}(x).
\end{equation}

One can associate the height profile $h(t,y)$ to the \mbox{$q$-TASEP} in the way as usual for TASEP as follows.
The height difference $h(t,y+1)-h(t,y)$ is either $-1$ or $+1$ if there is a particle at position $y$ at time $t$ or not respectively.
The initial profile is set to be $h(0,y)=|y|$.
Then the relation
\begin{equation}\label{hHrelation}
h(t,y)=2H(t,y+1)+y
\end{equation}
clearly holds and the above statement \eqref{currentrandomin} along with Theorem~\ref{thm:scalingtheory} below can easily be rephrased in terms of the height profile $h(t,y)$.
\end{remark}

Next we show that the KPZ scaling theory conjecture is satisfied for \mbox{$q$-TASEP}.
To state the scaling conjecture, we introduce the notations and follow the formalization of~\cite{Spo13a}.
Suppose that $\eta\in\{-1,1\}^\Z$ is a configuration of the system where the component $\eta_j=h_{j+1}-h_j$ can be understood as the height difference between positions $j+1$ and $j$ in a growth model of a surface.
The Markov generator of the process is given by
\begin{equation}
Lf(\eta)=\sum_{j\in\Z} c_{j,j+1}(\eta)(f(\eta^{j,j+1})-f(\eta))
\end{equation}
where $\eta^{j,j+1}$ is the configuration obtained from $\eta$ by interchanging the slopes at $j$ and $j+1$.
Wedge initial condition is imposed, i.e.\ $h(j,0)=|j|$.
It is assumed that there is a family of spatially ergodic and time stationary measures $\mu_\varrho$ of the process $\eta(t)$ labelled by the average slope
\begin{equation}
\varrho=\lim_{a\to\infty}\frac1{2a+1}\sum_{|j|\le a} \eta_j.
\end{equation}

The average steady state current and the integrated covariance of the conserved slope field are given by
\begin{equation}
j(\varrho)=\mu_\varrho(c_{0,1}(\eta)(\eta_0-\eta_1)),\qquad A(\varrho)=\sum_{j\in\Z}(\mu_\varrho(\eta_0\eta_j)-\mu_\varrho(\eta_0)^2).
\end{equation}
The deterministic profile function of the surface can be computed by the Legendre transform $\phi(y)=\sup_{|\varrho|\le1}(y\varrho-j(\varrho))$.
Then
\begin{conj}[KPZ class,~\cite{Spo13a}]\label{conj:KPZ}
Suppose that $\phi$ is twice differentiable at $y$ with $\phi''(y)\neq0$ and set $\varrho=\phi'(y)$ for $|\varrho|<1$.
If $A(\varrho)<\infty$ and $\lambda(\varrho)=-j''(\varrho)\neq0$, then
\begin{equation}\label{scalingtheory}
\lim_{t\to\infty}\P\left(h(yt,t)-t\phi(y)\ge-(-\textstyle\frac12\lambda A^2)^{1/3}st^{1/3}\right)=F_{\rm GUE}(s).
\end{equation}
\end{conj}

The following theorem is a consequence of Theorem~\ref{thm:main}.
It provides a form of Theorem~\ref{thm:main} in terms of the current fluctuations and it confirms Conjecture~\ref{conj:KPZ}.
The result is proved in Section~\ref{s:scalingtheory}.
\begin{theorem}\label{thm:scalingtheory}
Consider the current $H(t,y)$ defined in Remark~\ref{rem:current} and let $\xi_\tau$ be defined through
\begin{equation}\label{currentrandomout}
H\left(\tau,\frac{f-1}\kappa\tau\right)=\frac\tau\kappa+\frac1{q^\theta\kappa-f+1}\frac{\chi^{1/3}}{\kappa^{1/3}\log q}\xi_\tau \tau^{1/3}.
\end{equation}
If $q^\theta\le1/2$, then
\begin{equation}\label{rescheightconv}
\lim_{\tau\to\infty}\P(\xi_\tau<x)=F_{\rm GUE}(x).
\end{equation}
Further, the convergence \eqref{rescheightconv} confirms the KPZ scaling theory conjecture for \mbox{$q$-TASEP}.
\end{theorem}

\begin{corollary}\label{cor:density}
The expressions in Proposition~\ref{prop:lln} for the local particle density after time $\tau$ at position $(f-1)\tau/\kappa$ is confirmed for $q^\theta\le1/2$ without assuming local stationarity.
In particular, the particle density after time $\tau$ at position $p\tau$ for \mbox{$p\in(-(1-q),0]$} is recovered.
Under the local stationarity assumption, the current through the origin is the same as the one obtained at the end of Chapter 3 in~\cite{BC11}.
\end{corollary}

\begin{remark}\label{rem:techcond}
The condition $q^\theta\le1/2$ is equivalent to $\theta\ge\log_q(1/2)$ and the function $\theta\mapsto f(q,\theta)$ is decreasing.
But for $\theta=\log_q(1/2)$, $f-1=\sum_{k=1}^\infty q^{2k}/(2-q^k)^2\geq 0$ for any $q\in (0,1)$.
Therefore, for any $q\in (0,1)$, for the $\theta$ given by $(f-1)/\kappa=p$, the condition $q^\theta\le1/2$ holds for all $p\le0$.
However we expect Corollary~\ref{cor:density} to hold for any \mbox{$p\in(-(1-q),1)$}.
\end{remark}

\begin{remark}
Recently after the submission of the present paper, the technical condition $q^\theta\le1/2$ has been removed in~\cite{B14}.
\end{remark}

\section{Hydrodynamic limit}\label{s:hydro}

The goal of the present section is to show that under the local stationarity assumption, the law of large numbers \eqref{lln} holds for the particle positions.
In particular, we give the proof of Proposition~\ref{prop:lln} below.
The law of large numbers \eqref{lln} also follows from Theorem~\ref{thm:main}, but we give a direct argument here without the whole asymptotic analysis of the problem.

\begin{proof}[Proof of Proposition~\ref{prop:lln}]$ $
\begin{enumerate}
\item
If the gaps between particles have $q$-geometric distribution with parameter $\alpha$ given in \eqref{qgeodistr}, then under the local stationarity assumption, the particle density is given by
\begin{equation}\label{rhowithG}
\rho=\frac1{1+\E({\rm gap})}.
\end{equation}
The expectation above can be computed as follows:
\begin{equation}\label{EG}
\E({\rm gap})=(\alpha;q)_\infty\alpha\frac{\d}{\d\alpha}\sum_{k=0}^\infty\frac{\alpha^k}{(q;q)_k} =(\alpha;q)_\infty\alpha\frac{\d}{\d\alpha}\frac1{(\alpha;q)_\infty}
=\alpha\sum_{k=0}^\infty\frac{q^k}{1-\alpha q^k}.
\end{equation}
Using the series representation of the $q$-digamma function \eqref{Psiseries} on the right-hand side of \eqref{EG} and substituting it into \eqref{rhowithG} leads to $\rho$ in \eqref{rhojrho}.

The particle current in local stationarity is the product of the particle density and the expected rate of jump.
The latter is equal to
\begin{equation}
\E\left(1-q^{\rm gap}\right)=(\alpha;q)_\infty\sum_{k=0}^\infty(1-q^k)\frac{\alpha^k}{(q;q)_k} =\alpha(\alpha;q)_\infty\sum_{k=1}^\infty\frac{\alpha^{k-1}}{(q;q)_{k-1}}=\alpha.
\end{equation}
The formula for the current $j(\rho)$ in \eqref{rhojrho} now follows.

\item
From the derivatives of the two sides of \eqref{solutionPDE} with respect to $t$ and $\theta$, one can express
\begin{equation}\label{rhot}
\rho_t\left(t,\frac{f-1}\kappa t\right)=\frac1t \frac{f-1}\kappa \left(\frac{\d}{\d\theta}\frac{f-1}\kappa\right)^{-1} \frac{\Psi_q'(\theta)\log q}{(\log q+\log(1-q)+\Psi_q(\theta))^2}.
\end{equation}
Using the fact that the corresponding $j(\rho)$ is given by \eqref{rhojrho}, a similar calculation as above shows that
\begin{multline}\label{jx}
j_x\left(t,\frac{f-1}\kappa t\right)=\frac1t \left(\frac{\d}{\d\theta}\frac{f-1}\kappa\right)^{-1}\\
\times\left(\frac{q^\theta\log q}{\log q+\log(1-q)+\Psi_q(\theta)}-\frac{q^\theta\Psi_q'(\theta)\log q}{(\log q+\log(1-q)+\Psi_q(\theta))^2}\right).
\end{multline}
Finally, \eqref{defkappa}--\eqref{deff} imply that the sum of \eqref{rhot} and \eqref{jx} is $0$.

\item
Is follows immediately from the previous part of the proposition.

\item
What remains to show is that the number of particles to the right of position $(f-1)\tau/\kappa$ after time $\tau$ is about $\tau/\kappa$, i.e.\ the integral of the particle density between $(f-1)/\kappa$ and $1$ is $1/\kappa$.
By parameterizing the integration interval by $\tilde\theta$ that ranges from $\theta$ to $0$, the integral can be written as
\begin{equation}\label{densityintegral}
\int_0^\theta-\frac{\log q}{\log q+\log(1-q)+\Psi_q(\tilde\theta)}\left(\frac{\d}{\d\tilde\theta}\frac{f-1}\kappa\right)(\tilde\theta)\,\d\tilde\theta.
\end{equation}
By \eqref{rho1} and \eqref{implicitderivative}, the negative ratio in the integrand of \eqref{densityintegral} can be rewritten as the ratio of the derivatives on the left-hand side of \eqref{implicitderivative} taken at $\tilde\theta$.
This yields that the integral \eqref{densityintegral} is equal to $1/\kappa$ as required.
\end{enumerate}
\end{proof}

\section{Finite time formula}\label{s:finite}
In order to introduce the kernel that appears in our finite time formula, we define some integration contours in the complex plane. Some of them are also shown on Figure~\ref{fig:contours}.
\begin{figure}
\begin{center}
\def\svgwidth{180pt}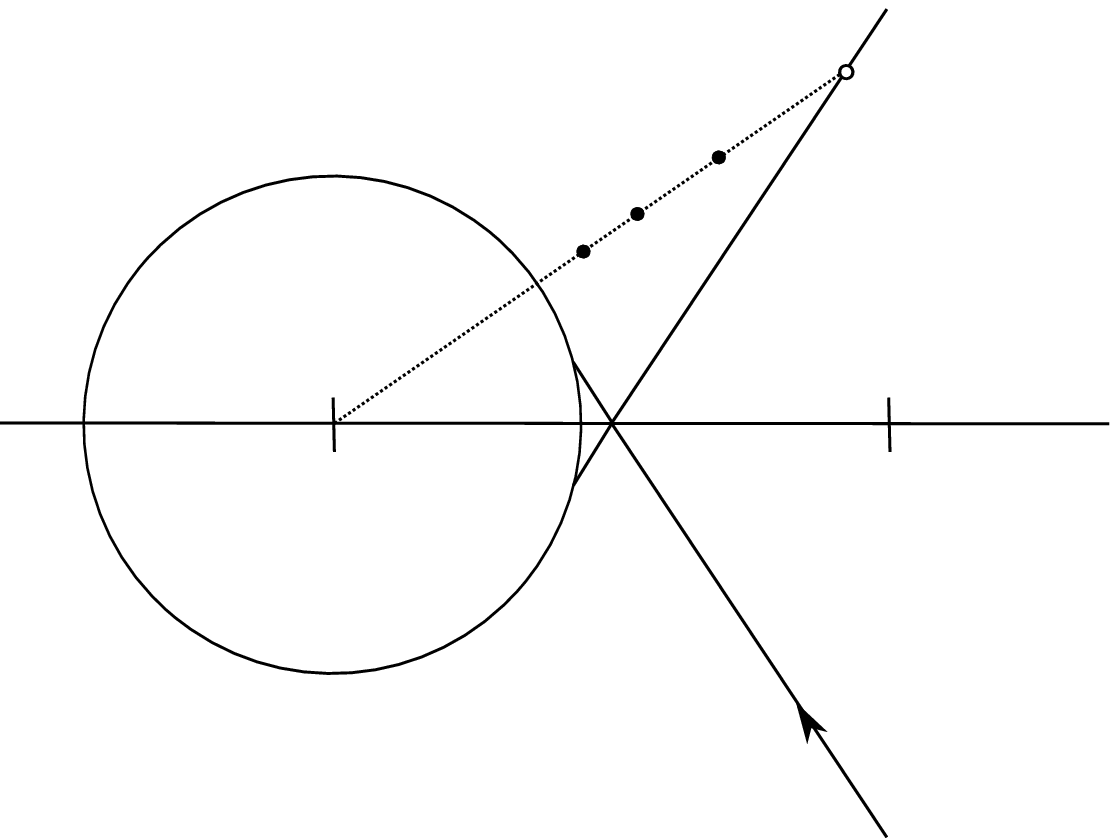
\hspace*{10pt}
\raisebox{10pt}{\def\svgwidth{180pt}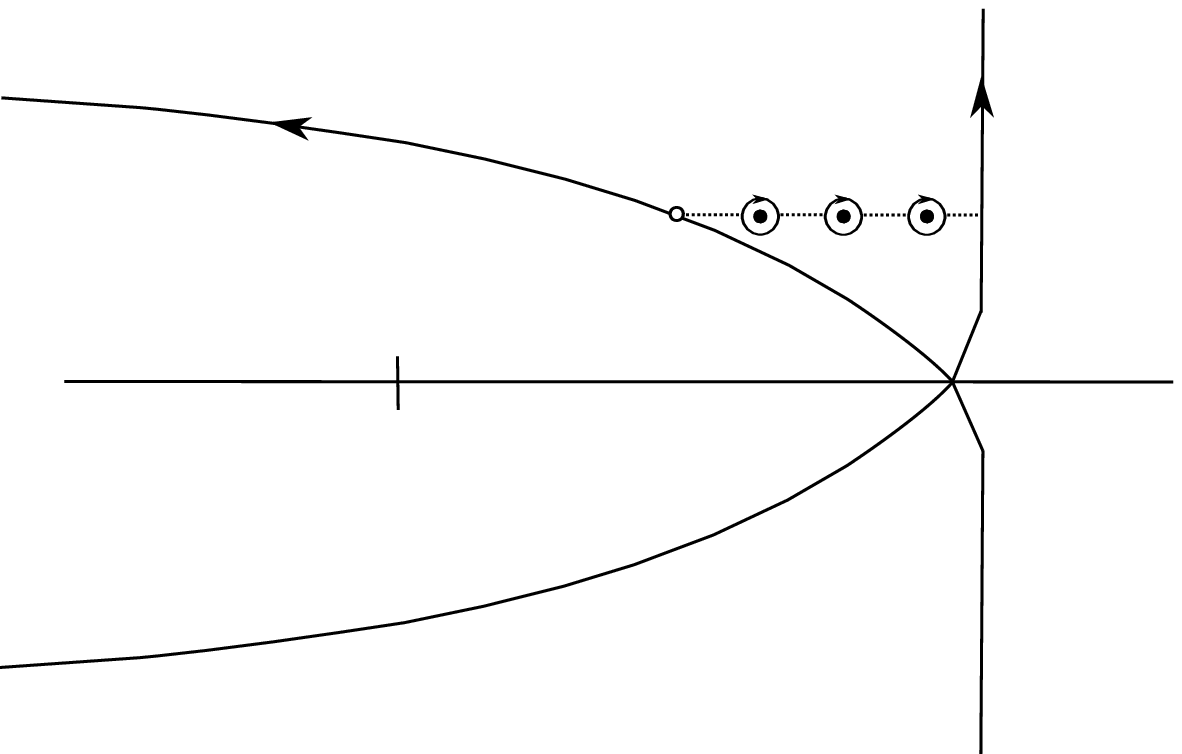}
\end{center}
\caption{The integration contours of the kernels.
The right-hand side is the image of the left-hand side under the map $x\mapsto\log_q x$:
$\C_\varphi=\log_q\wt\C_\varphi$; the circle around the origin on the left-hand side is the inverse image of the vertical line of $\D_W$ with the small modifications at $\theta$ and at $q^\theta$;
the images of the residues are also indicated.
\label{fig:contours}}
\end{figure}

\begin{definition}[Contours]\label{def:contours}
Let us fix a $q\in(0,1)$ and a $\theta>0$.
For a $\varphi\in(0,\pi/2]$, we define the contour $\wt\C_\varphi=\{q^\theta+e^{\I\varphi\sgn(y)}|y|:y\in\R\}$.
Let $\C_\varphi$ be its image under the map $x\mapsto\log_qx$, that is, $\C_\varphi=\{\log_q(q^\theta+e^{\I\varphi\sgn(y)}|y|):y\in\R\}$.
See Figure~\ref{fig:contours}.

For every $w\in\wt\C_\varphi$, define the contour $\wt\D_w$ that goes by straight lines from $R-\I \infty$ to $R-\I d$, to $1/2-\I d$, to $1/2+\I d$, to $R+\I d$ and to $R+\I\infty$
where $R,d>0$ are such that the following holds: let $b=\pi/4-\varphi/2$, then $\arg(w(q^s-1))\in(\pi/2+b,3\pi/2-b)$; and $q^sw$, $s\in\wt\D_w$ stays to the left of $\wt\C_\varphi$.

Let $\sigma>0$ be a small but fixed number, its value to be chosen later.
For $W\in\C_\varphi$, let $d>0$ sufficiently small and define the contour $\D_W$ as the previous contour shifted by $W$.
More precisely, $\D_W$ consists of straight lines from $\theta+\sigma-\I\infty$ to $\theta+\sigma+\I(\Im(W)-d)$, to $W+1/2+\I(\Im(W)-d)$, to $W+1/2+\I(\Im(W)+d)$,
to $\theta+\sigma+\I(\Im(W)+d)$ and to $\theta+\sigma+\I\infty$ such that $\D_W$ does not intersect $\C_\varphi$.
\end{definition}

The starting formula for the calculations of the present paper is a consequence of Theorem 4.13 in~\cite{BCF12}.

\begin{theorem}\label{thm:starting}
Let $\zeta\in\mathbb C\setminus\R_+$, then
\begin{equation}\label{startingid}
\E\left(\frac{1}{\left(\zeta q^{X_N(\tau)+N};q\right)_{\infty}}\right) = \det(\id-\wt K_{\zeta})_{L^2(\wt\C_\varphi)}
\end{equation}
for any $\varphi\in(0,\pi/4)$.
The operator $\wt K_{\zeta}$ is defined in terms of its integral kernel
\begin{equation}\label{defKtil}
\wt K_{\zeta}(w,w')=\frac1{2\pi\I}\int_{\wt\D_w}\d s\,\Gamma(-s)\Gamma(1+s)(-\zeta)^s g_{w,w'}(q^s),
\end{equation}
where
\begin{equation}\label{gwwprimeeqn}
g_{w,w'}(q^s) = \frac{\exp(\tau w(q^{s}-1))}{q^s w - w'} \left(\frac{(q^{s}w;q)_{\infty}}{(w;q)_{\infty}}\right)^N
\end{equation}
and the contours $\wt\C_\varphi$ and $\wt\D_w$ are as in Definition~\ref{def:contours}.
\end{theorem}

To see how Theorem~\ref{thm:starting} above follows from Theorem 4.13 of~\cite{BCF12}, the key ingredient is the following proposition (see Remark 3.3.4 in~\cite{BC11}).
The Macdonald measure with specialization $(1,\dots,1;\tau)$ is supported on partitions $(\lambda_1\ge\lambda_2\ge\dots\ge\lambda_N\ge0)$ of length at most $N$
if the number of $1$s in the specialization is $N$.

\begin{proposition}\label{prop:MM=qTASEP}
Let $q\in[0,1)$, $\tau>0$ and $N$ integer be fixed.
The marginal $\lambda_N$ under the Macdonald measure with parameters $q$ and $t=0$ under the specialization $(1,\dots,1;\tau)$ where the number of $1$s in the specialization is $N$
equals in distribution to $X_N(\tau)+N$ under the \mbox{$q$-TASEP} law.
\end{proposition}

\begin{proof}[Proof of Theorem~\ref{thm:starting}]
Starting from Theorem 4.13 in~\cite{BCF12} with $\wt a_i=1$ for \mbox{$i=1,\dots,N$}, we can apply Proposition~\ref{prop:MM=qTASEP} to get Theorem~\ref{thm:starting}.
\end{proof}

For later use, we first do the following change of variables in \eqref{startingid}--\eqref{defKtil}:
\begin{equation}\label{wWsZ}
w=q^W,\quad w'=q^{W'},\quad s+W=Z.
\end{equation}
We obtain the kernel
\begin{equation}\label{Ktil2}
\wh K_\zeta(W,W')=\frac{q^W\log q}{2\pi\I}\int_{\D_W}\frac{\d Z}{q^Z-q^{W'}}\frac\pi{\sin(\pi(W-Z))}
\frac{(-\zeta)^Z\exp(\tau q^Z+N\log(q^Z;q)_\infty)}{(-\zeta)^W\exp(\tau q^W+N\log(q^W;q)_\infty)}
\end{equation}
where the contour for $W$ and $W'$ is now $\C_\varphi$ for some $\varphi\in(0,\pi/4)$ and the $Z$-contour can be chosen to be $\D_W$ as in Definition~\ref{def:contours},
since we do not cross any singularity of the integrand and the $Z$-integral remains bounded due to the exponential decay of the sine inverse along a vertical line.

Note that $\D_W$ can be replaced by a vertical line and sufficiently small circles around the residues coming from the sine at $W+1,W+2,\dots,W+k_W$ where $k_W$ is the number of residues.
More precisely, for fixed but arbitrarily small $\sigma>0$ and for each $W\in\C_\varphi$, the corresponding vertical line will be either $\theta+\sigma+\I\R$ or $\theta+2\sigma+\I\R$
such that it is at distance at least $\sigma/2$ from the closest pole of the sine inverse.
It is also possible to modify the vertical line in a small neighbourhood of $\theta$ and replace it by the wedge $\{\theta+e^{\I\varphi\sgn(y)}|y|:y\in[-\delta,\delta]\}$.
We will refer also to this new contour as $\D_W$ which now consists of the vertical line perturbed at $\theta$ and the small circles.

The contours $\C_\varphi$ and $\D_W$ for the kernel $\wh K_\zeta$ are on the right-hand side of Figure~\ref{fig:contours}.
The left-hand side of the figure is the inverse image of these contours under the map $x\mapsto\log_q x$.
After transforming the Fredholm determinant to the contours shown on the right-hand side of Figure~\ref{fig:contours},
we will mostly work with these variables, but a part of the argument can be seen more easily in terms of the variables on the left-hand side of the figure.

\section{Convergence of $q$-Laplace transform}\label{s:qLaplace}

Let us choose
\begin{equation}\label{defzeta}
\zeta=-q^{-fN-cN^{2/3}+\beta_x\frac{N^{1/3}}{\log q}}\in\mathbb C\setminus\R_+
\end{equation}
where
\begin{equation}\label{defbeta}
\beta_x=c^2\frac{(\log q)^4}{4\chi}-\chi^{1/3}x.
\end{equation}
With the definition \eqref{defxi} and the choice of $\zeta$ above, the left-hand side of \eqref{startingid} becomes
\begin{equation}\label{exprewrite}
\E\left(\frac1{\left(\zeta q^{X_N(\tau)+N};q\right)_{\infty}}\right)
=\E\Bigg(\frac1{\big(-q^{\frac{\chi^{1/3}}{\log q}N^{1/3}(\xi_N-x)};q\big)_\infty}\Bigg).
\end{equation}
The function which appears on the right-hand side under the expectation above, has the following asymptotic property.

\begin{lemma}\label{lemma:unifconv}$ $
\begin{enumerate}
\item
The function
\begin{equation}
f_t(y)=\frac1{(-q^{yt};q)_\infty}
\end{equation}
is increasing for all $t>0$ and it is decreasing for all $t<0$.
\item
For each $\delta>0$, $f_t(y)\to\id(y>0)$ converges uniformly on $y\in\R\setminus[-\delta,\delta]$ as $t\to\infty$.
\end{enumerate}
\end{lemma}

\begin{proof}
By the definition of the $q$-Pochhammer symbol \eqref{defqpochhammer},
the factor $\frac1{1+q^{yt+k}}$ increases in $y$ for positive values of $t$ and decreases for negative values of $t$ for each $k$.
This proves the monotonicity.

The factor $\frac1{1+q^{yt+k}}$ is uniformly close to $1$ or to $0$ on $y\in(\delta,\infty)$ or on $y\in(-\infty,-\delta)$ respectively for each $k$ if $t$ is large enough.
Hence one can easily get the uniform convergence as stated in the lemma.
\end{proof}

We can use the next elementary probability lemma.
\begin{lemma}[Lemma~4.1.39 of~\cite{BC11}]\label{lemma:problemma1}
Consider a sequence of functions $(f_n)_{n\geq 1}$ mapping $\R\to [0,1]$ such that for each $n$, $f_n(y)$ is strictly decreasing in $y$ with a limit of $1$ at $y=-\infty$ and $0$ at $y=\infty$,
and for each $\delta>0$, on $\R\setminus[-\delta,\delta]$, $f_n$ converges uniformly to $\id(y<0)$. Consider a sequence of random variables $X_n$ such that for each $r\in\R$,
\begin{equation}
\E[f_n(X_n-r)] \to p(r)
\end{equation}
and assume that $p(r)$ is a continuous probability distribution function.
Then $X_n$ converges weakly in distribution to a random variable $X$ which is distributed according to \mbox{$\P(X<r) = p(r)$}.
\end{lemma}

\begin{proof}[Proof of Theorem~\ref{thm:main}]
The sequence of functions
\[f_N(y)=\frac1{\left(-q^{\frac{\chi^{1/3}}{\log q}N^{1/3}y};q\right)_\infty}\]
satisfies the requirements of Lemma~\ref{lemma:problemma1} by Lemma~\ref{lemma:unifconv} since $\log q<0$.
By Theorem~\ref{thm:Fredholmconv} below and the equation \eqref{startingid}, \eqref{exprewrite} also converges to $F_{\rm GUE}(x)$ the Tracy--Widom distribution function,
hence we can use Lemma~\ref{lemma:problemma1} to conclude that the sequence $\xi_N$ converges in law to the Tracy--Widom distribution.
\end{proof}

\section{Asymptotic analysis}\label{s:asympt}

This section is devoted to prove the next convergence result for Fredholm determinants which is the key fact for the Tracy--Widom limit of the current fluctuation of \mbox{$q$-TASEP}.

\begin{theorem}\label{thm:Fredholmconv}
Let $x\in\R$ be fixed and choose $\zeta$ according to \eqref{defzeta}.
Then
\begin{equation}
\det(\id-\wt K_\zeta)_{L^2(\wt\C_\varphi)}\to F_{\rm GUE}(x)
\end{equation}
as $N\to\infty$.
\end{theorem}

If we substitute \eqref{deftau} and \eqref{defzeta} into \eqref{Ktil2}, then
\begin{equation}
\det(\id-\wt K_\zeta)_{L^2(\wt\C_\varphi)} = \det(\id-K_x)_{L^2(\C_\varphi)},
\end{equation}
where
\begin{equation}\label{Ktil3}
K_x(W,W')=\frac{q^W\log q}{2\pi\I}\int_{\D_W}\frac{\d Z}{q^Z-q^{W'}}\frac\pi{\sin(\pi(W-Z))}
\frac{e^{Nf_0(Z)+N^{2/3}f_1(Z)+N^{1/3}f_2(Z)}}{e^{Nf_0(W)+N^{2/3}f_1(W)+N^{1/3}f_2(W)}}
\end{equation}
with
\begin{align}
f_0(Z)&=-f(\log q)Z+\kappa q^Z+\log(q^Z;q)_\infty,\\
f_1(Z)&=-c(\log q)Z+cq^{Z-\theta},\\
f_2(Z)&=\beta_xZ
\end{align}
and $\beta_x$ as in \eqref{defbeta}.
Then Theorem~\ref{thm:Fredholmconv} is proved via the following series of propositions.

\begin{proposition}\label{prop:kerneldeformation}
For fixed $q\in(0,1)$, $\theta>0$ and $N$ large enough, the contour $\C_\varphi$ with $\varphi\in(0,\pi/4)$ for the kernel $K_x$ can be extended to any $\varphi\in(0,\pi/2)$
without effecting the Fredholm determinant $\det(\id-K_x)_{L^2(\C_\varphi)}$.
\end{proposition}

\begin{proposition}\label{prop:localization}
For any fixed $\delta>0$ and $\varepsilon>0$ small enough, there is an $N_0$ such that
\begin{equation}
\left|\det(\id-K_x)_{L^2(\C_\varphi)}-\det(\id-K_{x,\delta})_{L^2(\C_\varphi^\delta)}\right|<\varepsilon
\end{equation}
for all $N>N_0$ where $\C_\varphi^\delta=\C_\varphi\cap\{w:|w-\theta|\le\delta\}$ and
\begin{equation}\label{defKxdelta}
K_{x,\delta}(W,W')=\frac{q^W\log q}{2\pi\I}\int_{\D_W^\delta}\frac{\d Z}{q^Z-q^{W'}}\frac\pi{\sin(\pi(W-Z))}
\frac{e^{Nf_0(Z)+N^{2/3}f_1(Z)+N^{1/3}f_2(Z)}}{e^{Nf_0(W)+N^{2/3}f_1(W)+N^{1/3}f_2(W)}}
\end{equation}
and $\D_W^\delta=\D_W\cap\{z:|z-\theta|\le\delta\}$.
\end{proposition}

It follows immediately from the Cauchy theorem that in the Fredholm determinant $\det(\id-K_x)_{L^2(\C_\varphi^\delta)}$,
the contour $\C_\varphi^\delta$ can be deformed to $\{\theta+e^{\I(\pi-\wt\varphi)\sgn(y)}|y|,y\in[-\delta,\delta]\}$ with another angle $\wt\varphi$ close to $\pi/2$
chosen such that the two contours have the same endpoints.
With a slight abuse of notation, we also use $\varphi$ for the new angle $\wt\varphi$.
Similarly, the integration contour for $Z$ in \eqref{defKxdelta} can be deformed without changing the integral to the contour
$\D_\varphi^\delta=\{\theta+e^{\I\varphi\sgn(t)}|t|:t\in[-\delta,\delta]\}$ if $\sigma$ in Definition~\ref{def:contours} is chosen such that the endpoints of the two contours are the same.

Let us consider the rescaled kernel
\begin{equation}\label{defKxd}
K_{x,\delta}^N(w,w')=N^{-1/3}K_{x,\delta N^{1/3}}(\theta+wN^{-1/3},\theta+w'N^{-1/3})
\end{equation}
with $\D_W^\delta$ replaced by $\D_\varphi^\delta$ in the definition \eqref{defKxdelta} of the kernel $K_{x,\delta N^{1/3}}$.
Let us also define the contour $\C_{\varphi,L}=\{e^{\I(\pi-\varphi)\sgn(y)}|y|,y\in[-L,L]\}$ for any $L>0$ including also $L=\infty$ in which case we mean the infinite contour with $y\in\R$ in the definition.
Note that by the above deformation argument and by simple rescaling,
\begin{equation}
\det(\id-K_{x,\delta})_{L^2(\C_\varphi^\delta)}=\det(\id-K_{x,\delta}^N)_{L^2(\C_{\varphi,\delta N^{1/3}})}.
\end{equation}
Therefore, the proof of Theorem~\ref{thm:Fredholmconv} is a consequence of the following assertions which are proved in subsequent sections.

\begin{proposition}\label{prop:kernelconv}
Let $\varphi\in(0,\pi/2)$ be sufficiently close to $\pi/2$ and let $\varepsilon>0$ be fixed.
There is a small $\delta>0$ and an $N_0$ such that for any $N>N_0$,
\begin{equation}
\left|\det(\id-K_{x,\delta}^N)_{L^2(\C_{\varphi,\delta N^{1/3}})}-\det(\id-K_{x,\delta N^{1/3}}')_{L^2(\C_{\varphi,\delta N^{1/3}})}\right|<\varepsilon
\end{equation}
where
\begin{equation}\label{kernelK'}
K_{x,L}'(w,w')=\frac1{2\pi\I}\int_{\D_{\varphi,L}}\frac{\d z}{(z-w')(w-z)}\frac{e^{\chi z^3/3+c(\log q)^2z^2/2+\beta_x z}}{e^{\chi w^3/3+c(\log q)^2w^2/2+\beta_x w}}.
\end{equation}
on the contour $\D_{\varphi,L}=\{e^{\I\varphi\sgn(y)}|y|:y\in[-L,L]\}$ for $L>0$ including also $L=\infty$ in which case we mean the infinite contour with $y\in\R$ in the definition and the corresponding kernel.
\end{proposition}

\begin{proposition}\label{prop:kernelextend}
With the notation as above,
\begin{equation}
\det(\id-K_{x,\delta N^{1/3}}')_{L^2(\C_{\varphi,\delta N^{1/3}})}\to\det(\id-K_{x,\infty}')_{L^2(\C_{\varphi,\infty})}
\end{equation}
as $N\to\infty$.
\end{proposition}

\begin{proposition}\label{prop:rewritekernel}
We can rewrite the Fredholm determinant
\begin{equation}
\det(\id-K_{x,\infty}')_{L^2(\C_{\varphi,\infty})}=\det(\id-K_{\Ai,x})_{L^2(\R_+)}=F_{\rm GUE}(x)
\end{equation}
where
\begin{equation}
K_{\Ai,x}(a,b)=\int_0^\infty\d\lambda\Ai(x+a)\Ai(x+b)
\end{equation}
and $F_{\rm GUE}$ is the GUE Tracy--Widom distribution function.
\end{proposition}

\begin{remark}
The reason for extending the Fredholm determinant formula for \mbox{$\varphi\in(0,\pi/2)$} in Proposition~\ref{prop:kerneldeformation} is that
for $\varphi<\pi/4$ the contour $\C_\varphi$ is not steep descent for $-\Re(f_0)$.
Not only the proof of Lemma~\ref{lemma:steep} breaks down, but it happens in general for $\varphi<\pi/4$ that there are higher values of $-\Re(f_0)$ along $\C_\varphi$ than at $\theta$.
\end{remark}

\subsection{Steep descent contours}

Using the values of $\kappa$, $f$ and $\chi$ from Definition~\ref{def:kfa} and the formulas from Definition~\ref{def:qdigamma},
one can rewrite the derivatives of the functions $f_0$, $f_1$ and $f_2$ in terms of $q$-digamma functions.
One has the following Taylor expansions of these functions around $\theta$:
\begin{align}
f_0(Z)&=f_0(\theta)+\frac\chi3(Z-\theta)^3+\O((Z-\theta)^4),\label{f0Taylor}\\
f_1(Z)&=f_1(\theta)+\frac{c(\log q)^2}2(Z-\theta)^2+\O((Z-\theta)^3),\label{f1Taylor}\\
f_2(Z)&=f_2(\theta)+\beta_x(Z-\theta).\label{f2Taylor}
\end{align}

\begin{figure}
\begin{center}
\includegraphics[width=250pt]{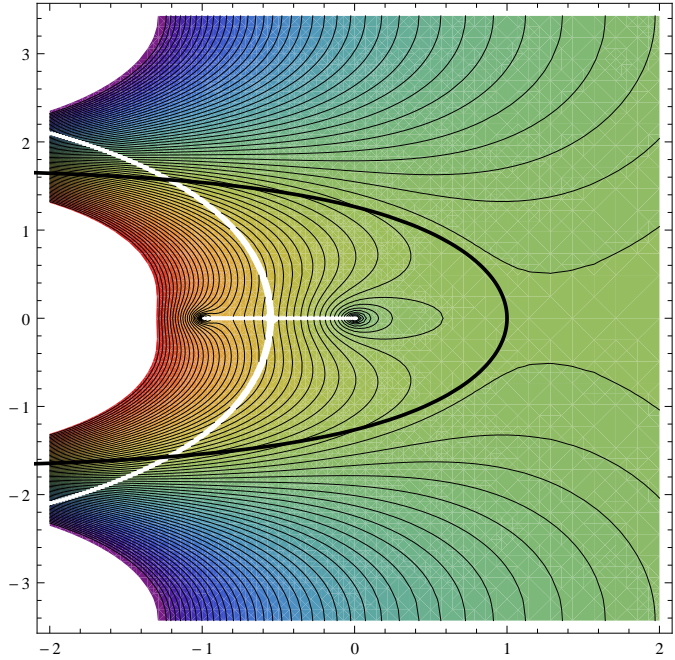}
\end{center}
\caption{The contour plot of the function $\Re(f_0)$ and the integration contour $\C_\varphi$ in thick black.
The function $\Re(f_0)$ is periodic in the vertical direction, one period is shown.
The steep descent property of $\C_\varphi$ can be seen.\label{fig:function_and_curves}}
\end{figure}

For the asymptotic analysis, the behaviour of the function $\Re(f_0)$ is crucial, since it governs the leading term of the exponent in the kernel, see \eqref{Ktil3}.
A contour plot of this function is shown on Figure~\ref{fig:function_and_curves}.
An interesting feature of $\Re(f_0)$ is that it is vertically periodic with period $\I2\pi/|\log q|$.
As it can be seen on the figure, the imaginary part of a point along the contour $\C_\varphi$ remains bounded,
therefore the convergence of the Fredholm determinant does not come from the decay of the sine inverse but from that of the term $q^W$ in $f_0(W)$.
On the other hand, because of the vertical periodicity of $\Re(f_0)$, to show the convergence of the integral in $Z$, we need to use the decay from the sine inverse.
A similar issue of periodicity does not appear in~\cite{BCF12}, hence it was not needed at the corresponding point of the analysis for the semi-discrete directed polymers to use the decay from the sine inverse.
However this decay was used in the derivation of the pre-asymptotic Laplace transform formula for the semi-discrete directed polymer partition function in~\cite{BCF12} and also in~\cite{BC11}.

The derivative $f_0'$ can be expressed as
\begin{equation}\label{f0'}
f_0'(Z)=\frac{\Psi_q'(\theta)}{\log q}(q^{Z-\theta}-1)+\Psi_q(\theta)-\Psi_q(Z).
\end{equation}
It follows by derivation from Definition~\ref{def:qdigamma} that the $q$-digamma function and its derivative have the useful series representations
\begin{align}
\Psi_q(Z)&=-\log(1-q)+\log q\sum_{k=0}^\infty \frac{q^{Z+k}}{1-q^{Z+k}},\label{Psiseries}\\
\Psi_q'(Z)&=(\log q)^2\sum_{k=0}^\infty \frac{q^{Z+k}}{(1-q^{Z+k})^2}.\label{Psi'series}
\end{align}
Combining \eqref{f0'} with \eqref{Psiseries}--\eqref{Psi'series}, we get
\begin{equation}\label{f0'series}
f_0'(Z)=-\log q \sum_{k=0}^\infty \frac{q^{2k}(q^\theta-q^Z)^2}{(1-q^{\theta+k})^2(1-q^{Z+k})}.
\end{equation}
Note that the series \eqref{kappafseries} for $\kappa$ and $f$ also follow from \eqref{Psiseries}--\eqref{Psi'series}.

\begin{lemma}\label{lemma:steep}$ $
\begin{enumerate}
\item
If $\varphi\in[\pi/4,\pi/2]$ and $q^\theta\le1/2$, then the contour $\C_\varphi$ given in Definition~\ref{def:contours} is steep descent for the function $-\Re(f_0)$
in the sense that $-\Re(f_0)$ attains its maximum at $\theta$ and it is decreasing along $\C_\varphi$ towards both ends.

\item
The function $\Re(f_0)$ is periodic on $\{\theta+\gamma+\I t,t\in\R\}$ with period $2\pi/|\log q|$.
The contour $\{\theta+\gamma+\I t,t\in[\pi/\log q,-\pi/\log q]\}$ for any $\gamma\ge0$ is steep descent for the function $\Re(f_0)$ in the same sense as above.
\end{enumerate}
\end{lemma}

\begin{proof}
\begin{enumerate}
\item
Consider $W(s)=\log_q(q^\theta+e^{\I\varphi}s)$ for $s\ge0$ which parameterizes the part of $\C_\varphi$ above the real axis.
It is elementary to see using \eqref{f0'series} that
\begin{equation}\begin{aligned}
&-\frac{\d}{\d s}\Re(f_0(W(s)))\\
&\qquad=\Re\sum_{k=0}^\infty \frac{q^{2k}e^{3\I\varphi}s^2}{(1-q^{\theta+k})^2(1-q^{\theta+k}-e^{\I\varphi}sq^k)(q^\theta+e^{\I\varphi}s)}\\
&\qquad=\sum_{k=0}^\infty \frac{q^{2k}s^2(-q^ks^2\cos\varphi+(1-2q^{\theta+k})s\cos2\varphi+q^\theta(1-q^{\theta+k})\cos3\varphi)}
{(1-q^{\theta+k})^2|1-q^{\theta+k}-e^{\I\varphi}sq^k|^2|q^\theta+e^{\I\varphi}s|^2}.\label{f0W}
\end{aligned}\end{equation}
It is clear that for the given choice of $\varphi$, $q$ and $\theta$, all the three terms in the numerator of \eqref{f0W} are non-positive,
hence the derivative \eqref{f0W} is non-positive (and negative except for $s=0$).
The proof for negative values of $s$ is similar.

\item
Periodicity is obvious.
Let $Z(t)=\theta+\gamma+\I t$.
Then it can be seen that
\begin{equation}\label{f0Z}
\frac{\d}{\d t}\Re(f_0(Z(t)))
=-\sin(t\log q)\log q\sum_{k=0}^\infty q^{\theta+\gamma+k}\left(\frac1{(1-q^{\theta+k})^2}-\frac1{|1-q^{\theta+\gamma+\I t+k}|^2}\right).
\end{equation}
As long as $\gamma\ge0$, the difference in parentheses on the right-hand side of \eqref{f0Z} is also non-negative (and positive except for $t=0$).
Keeping in mind that $\log q<0$, this proves the steep descent property.
\end{enumerate}
\end{proof}

\begin{remark}
Remark that for $\varphi=\pi/4$, Lemma~\ref{lemma:steep} holds for any $\theta>0$.
\end{remark}

\subsection{Contour deformation and localization}

In this section, we prove Propositions~\ref{prop:kerneldeformation} and~\ref{prop:localization} about the deformation and localization of the Fredholm determinant of the kernel $K_x$.
In both proofs we will parameterize the variable $W$ as
\begin{equation}\label{Wparam}
W(s)=\log_q(q^\theta+e^{\I\varphi\sgn(s)}|s|),
\end{equation}
but we only focus on the values $s\ge0$ since the argument for $s<0$ is the same by symmetry.

\begin{lemma}\label{lemma:expdecay}
The kernel $K_x(W(s),W')$ is exponentially small in $s$.
More precisely,
\begin{equation}\label{expdecay}
|K_x(W(s),W')|\le\exp(-cNs)
\end{equation}
for some constant $c>0$ which is uniformly positive if $\varphi$ is bounded away from $\pi/2$ and for all $N>N_0$ and $s>s_0$ with some thresholds $N_0$ and $s_0$.
\end{lemma}

\begin{proof}
We use the form \eqref{Ktil3} for the kernel $K_x(W,W')$ and we investigate its behaviour as $W=W(s)$ for $s$ large.
By \eqref{f0W}, for any fixed $\varphi\in[\pi/4,\pi/2)$, the derivative $-\frac{\d}{\d s}\Re(f_0(W(s)))$ converges to a negative constant
\begin{equation}
-c_\varphi=-\cos\varphi\sum_{k=0}^\infty \frac{q^k}{(1-q^{\theta+k})^2}.
\end{equation}
The constant $c_\varphi$ is uniformly positive as long as $\varphi$ is bounded away from $\pi/2$.
The derivatives of $\Re(f_1(W(s)))$ and $\Re(f_2(W(s)))$ are also at most constants, but since their prefactor is a lower power of $N$, they are negligible if $N$ is large enough.

Let $s$ be large and take $W=W(s)$ with \eqref{Wparam}.
We rewrite the integral over $\D_W$ as the integral over the vertical line $\theta+\delta+\I\R$ and some residues coming from the sine function in the denominator.
The integrand in $Z$ has an exponential decay in $Z$ along the vertical line due to the sine in the denominator for any $W$.
Hence by the observations above, the integral over the vertical part of the $Z$ contour can be bounded by $\exp(-cNs)$ for some $c>0$ and for $N$ and $s$ large enough.
Therefore it is enough to prove the same for the residues in order to establish \eqref{expdecay}.

We prove in what follow that the residues are each at most $\exp(-cNs)$, however the number of residues is only logarithmic in $s$.
The positions of the poles corresponding to $W(s)$ are $W(s)+1,\dots,W(s)+k_W$ where $k_W$ is the number of the poles (which might also be $0$).
Let us choose $s_j>0$ for $j=1,\dots,k_W$ such that $\Re W(s_j)=\Re W(s)+j$, i.e., $W(s_j)$ is the vertical projection of the $j$th pole to $\C_\varphi$, see Figure~\ref{fig:residue_approach}.
The strategy is to approach the residue at $W(s)+j$ in two steps: first we go along the contour $\C_\varphi$ until $W(s_j)$, then we proceed along a vertical line to the pole.
This enables us to use the information about $\Re(f_0)$ given in the proof of Lemma~\ref{lemma:steep}.

\begin{figure}
\begin{center}
\def\svgwidth{180pt}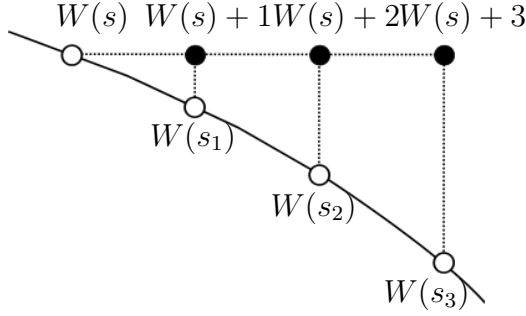
\end{center}
\caption{For a $W(s)\in\C_\varphi$, the poles are at positions $W(s)+1,W(s)+2,\dots$,
the points $W(s_1),W(s_2),\dots$ are the vertical projections of the positions of the poles to $\C_\varphi$.\label{fig:residue_approach}}
\end{figure}

We now show that $\Re(f_0(W(s))-f_0(W(s_1)))\asymp s$ where $\asymp$ means that the ratio of the two sides goes to a non-trivial constant as $s\to\infty$.
By the definition of $s_1$, we have $|q^{W(s)}|=q^{-1}|q^{W(s_1)}|$ where $q^{W(s)}$ is the image of of $W(s)$ on the left-hand side of Figure~\ref{fig:contours}.
Since the contour $\wt\C_\varphi$ is parameterized by arc length, for large $s$, $|q^{W(s)}|\asymp s$ and $|q^{W(s_1)}|\asymp s_1$, that is, $s-s_1\asymp s$.
Since for large $s$, $\frac{\d}{\d s}\Re(f_0(W(s)))\asymp-c_\varphi$, we also get that $\Re(f_0(W(s))-f_0(W(s_1)))\asymp s$.

Next we show that $\Re(f_0(W(s)+1))$ is even smaller than $\Re f_0(W(s_1))$.
For this end, we consider the function $\Re(f_0)$ along a vertical line, so we can use the derivative \eqref{f0Z} computed in the proof of Lemma~\ref{lemma:steep}.
There, it was used for $\gamma\ge0$, but if $s$ is large in our case, $\gamma$ is negative with large absolute value, hence the derivative of $\Re(f_0)$ is non-positive.
Therefore the contribution of the first residue is exponentially small.

The first part of this argument, i.e., $\Re(f_0(W(s))-f_0(W(s_j)))\asymp s$ remains true for all the poles $j=1,\dots,k_W$.
But there is a finite integer $L$ which does not depend on $s$ or $N$ such that the second half of the argument (that is, $\Re(f_0(W(s)+j)-f_0(W(s_j)))\le0$) works for all but the last $L$ poles.
The existence of such an $L$ is simply because for a fixed $q$ and $\theta$,
the sum on the right-hand side of \eqref{f0Z} is positive uniformly in $t\in\R$ if $\gamma$ is smaller than a fixed negative constant.
For those poles which have real part $\Re W(s)+j$ in the $L$ neighbourhood of $\theta$, $\Re(f_0(W(s)+j))$ might not be smaller than $\Re(f_0(W(s_j)))$,
but as we will see that the difference can be upper bounded by a fixed amount.
We again consider the derivative of $\Re(f_0)$ along a vertical line as in \eqref{f0Z}, but now in the domain $G_\varphi$
where $G_\varphi$ is $\{W:\Im(W)\in[\pi/\log q,-\pi/\log q]\}$ without the domain on the left-hand side of $\C_\varphi$ seen in the direction of its orientation.

It is enough to show that the derivative of $\Re(f_0)$ along a vertical line can be uniformly upper bounded on the domain $G_\varphi$,
since the distance of $W(s_j)$ and $W(s)+j$ is at most $2\pi/|\log q|$, the vertical period of $\Re(f_0)$.
In terms of Figure~\ref{fig:contours}, the image of the domain $G_\varphi$ on the left-hand side of the figure is on the left-hand side of the contour $\wt\C_\varphi$
which has a uniformly positive distance from $1$, that is, the difference in parentheses on the right-hand side of \eqref{f0Z} can be uniformly upper bounded.
Therefore we get that the difference $\Re(f_0(W(s)+j)-f_0(W(s_j)))$ is at most a bounded amount.

It may also happen that the last residue $W(s)+k_W$ on the left-hand side of the vertical part of $\D_W$
(which is at $\theta+\sigma+\I\R$ or $\theta+2\sigma+\I\R$, but $\sigma$ is as small as we want) has a larger real part than $\theta$.
In this case, we approach it from $W(s)$ as follows.
We first proceed to $\theta$ along $\C_\varphi$ where $\Re(f_0)$ decreases.
Then we continue on a short horizontal line which is not longer than $2\sigma$ until $\Re W(s)+k_W$, the derivative of $\Re(f_0)$ \eqref{f0'series} is bounded here.
Finally we go vertically to $W(s)+k_W$ where $\Re(f_0)$ is again non-increasing by the steep descent property.
Also in this case, $\Re(f_0)$ can only increase by a bounded amount.
Therefore if $s$ is large enough, $\Re(f_0(W(s_j))-f_0(W(s)))$ dominates and we get exponential decay of the kernel $K_x(W(s),W')$ in $N$ and $s$ as stated.
\end{proof}

\begin{remark}
In order to control the values of $\Re(f_0)$ at the residues, we follow the contour $\C_\varphi$ first and then go along a vertical line,
since the derivative of the function along a vertical line given in \eqref{f0Z} is relatively simple.
It could be a natural choice to look at the derivative of the function along a horizontal line, but unlike in~\cite{BCF12}, the function is not monotonic along these lines,
hence we found it more difficult to analyse.
\end{remark}

The proof of Proposition~\ref{prop:kerneldeformation} now easily follow from Lemma~\ref{lemma:expdecay}.

\begin{proof}[Proof of Proposition~\ref{prop:kerneldeformation}]
We give an integrable bound \eqref{kernelbound} on the kernel $K_x$ along the contour $\C_\varphi$.
Let $N$ be large and fixed such Lemma~\ref{lemma:expdecay} ensures that \eqref{expdecay} holds uniformly in $|s|\ge s_0$ for some threshold $s_0$.
If $s$ varies in the compact interval $|s|\le s_0$, the kernel $K(W(s),W')$ diverges logarithmically in the neighbourhood of $s=0$, but it is bounded otherwise.
The behaviour in the second variable of the kernel $K_x$ is similar, therefore we have
\begin{equation}\label{kernelbound}
|K_x(W(s),W(s'))|\le Ce^{-cNs}+C(\log|s|)_-(\log|s'|)_-.
\end{equation}
Hence the series expansion of the Fredholm determinant can be uniformly bounded by a convergent series as $\varphi$ varies in any given closed subinterval of $(0,\pi/2)$.
Therefore, by the dominated convergence theorem, one can perform the contour deformation in each term of the series expansion of the determinant yielding the validity of the contour deformation in the determinant.
\end{proof}

\begin{lemma}\label{lemma:residues}
For any $W\in\C_\varphi$, let $W+j$ be the positions of poles of the sine from which there is a residue contribution for $j=1,\dots,k_W$.
The angle $\varphi$ can be chosen so close to $\pi/2$ that $\Re(f_0(W)-f_0(W+j))\ge\varepsilon_\varphi$ for all $j=1,\dots,k_W$ with some $\varepsilon_\varphi>0$ which does not depend on $W$.
\end{lemma}

\begin{proof}
The strategy is similar to the proof of Lemma~\ref{lemma:expdecay}.
For a fixed \mbox{$W\in\C_\varphi$}, let \mbox{$W_j\in\C_\varphi$} be the unique point on the same side of the real axis such that \mbox{$\Re W_j=\Re W+j$} for $j=1,\dots,k_W$.
As in the proof of Lemma~\ref{lemma:expdecay}, if $W$ is far enough from $\theta$ along $\C_\varphi$,
then the decay along $\C_\varphi$ is much larger than the possible increase along the vertical line which is proved to be at most bounded,
hence $\Re(f_0(W)-f_0(W+j))\ge\varepsilon_\varphi$ holds for all $j=1,\dots,k_W$.

So it is enough to prove the lemma for $W$ in a compact neighbourhood of $\theta$.
In this compact set, the vertical distances $|W+j-W_j|$ can be uniformly lower bounded.
By the steep descent property stated in Lemma~\ref{lemma:steep}, $\Re(f_0(W)-f_0(W_j))>0$.
In what follows, we choose $\varphi$ such that also
\begin{equation}\label{vertdiff}
\Re(f_0(W_j)-f_0(W+j))\ge\varepsilon_\varphi>0
\end{equation}
which is enough to complete the proof.
There is a horizontal strip around the real axis such that if $\varphi\in[\pi/4,\pi/2]$, then there is no residue term for $W\in\C_\varphi$ in this strip, i.e., $k_W=0$,
because the position of the first pole at $W+1$ is already on the right-hand side of $D_W$.
This strip on the right-hand side of Figure~\ref{fig:contours} corresponds to a sector $\{w:\arg(w)\in[-\psi,\psi]\}$ on the left-hand side for some angle $\psi>0$
which can be chosen uniformly for $\varphi\in[\pi/4,\pi/2]$ (see Figure~\ref{fig:intersection}).
So it is enough to prove the lemma for the complement of the sector.

Since $W_j$ and $W+j$ are on the same vertical line $t\mapsto\theta+\gamma+\I t$ in the complex plane for some $\gamma<0$, in order to show \eqref{vertdiff},
we can use again the formula \eqref{f0Z} for the derivative of $\Re(f_0)$ along a vertical line.
For $t>0$, the summands in the sum on the right-hand side of \eqref{f0Z} are negative if
\begin{equation}\label{geoineq}
|1-q^{\theta+k}|<|1-q^{\theta+k}q^{\gamma+\I t}|.
\end{equation}
This is equivalent to $q^{\theta+k}q^{\gamma+\I t}$ being outside the circle of radius $1-q^{\theta+k}$ around $1$.
Let $S^k_\varphi$ be the domain on the left-hand side of $q^k\wt\C_\varphi$ intersected with the complement of the sector $\{w:\arg(w)\in[-\psi,\psi]\}$.
As illustrated on Figure~\ref{fig:intersection}, $\varphi$ can be chosen so close to $\pi/2$ that $S^k_\varphi$ has a positive distance from the circle of radius $1-q^{\theta+k}$ around $1$.
It makes it possible to ensure that the $k$th summand in \eqref{f0Z} is negative for $k=0,\dots,K$ and that they add up to something less than $-\epsilon$ with $\epsilon>0$.
By the exponential decay in the sum on the right-hand side of \eqref{f0Z},
we can find $K$ large enough that makes the tail smaller than $\epsilon/2$ in absolute value which ensure that the whole sum is uniformly negative.
(Increasing $K$ might require another $\varphi$ which is closer to $\pi/2$, but with this the sum of the first finitely many negative terms will not increase.)

Since the vertical distances have a lower bound, the argument gives a lower bound on $\Re(f_0(W_j)-f_0(W+j))$ which is independent of $W$.
The case when the position of the last pole has larger real part than $\theta$ can be similarly handled as in the proof of Lemma~\ref{lemma:expdecay},
since the difference can be at most $2\sigma$ which can be made arbitrarily small.
This completes the proof.

\begin{figure}
\begin{center}
\def\svgwidth{220pt}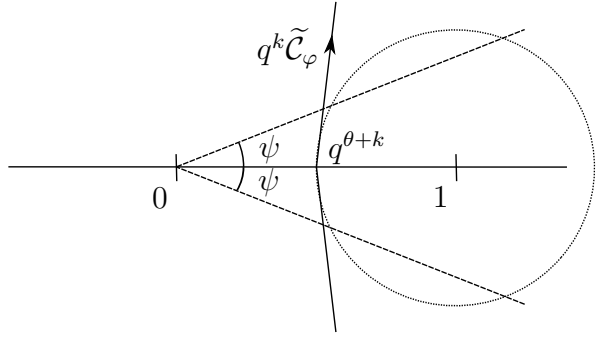
\end{center}
\caption{The angle $\varphi$ can be chosen so close to $\pi/2$ that the domain on the left-hand side of $q^k\wt\C_\varphi$
without the sector $\{w\in\wt\C_\varphi:\arg(w)\in[-\psi,\psi]\}$ does not intersect the circle of radius $1-q^{\theta+k}$ around $1$.
\label{fig:intersection}}
\end{figure}
\end{proof}

\begin{proof}[Proof of Proposition~\ref{prop:localization}]
We consider the Fredholm determinant expansion
\begin{equation}\label{Fredholmseries}
\det(\id-K_x)_{L^2(\C_\varphi)}=\sum_{k=0}^\infty \frac{(-1)^k}{k!} \int_\R\d s_1\dots\int_\R\d s_k \det\left(K_x(W(s_i),W(s_j))\frac{\d}{\d s_i}W(s_i)\right)_{i,j=1}^k.
\end{equation}
We use the bound \eqref{kernelbound} as in the proof of Proposition~\ref{prop:kerneldeformation} which is integrable and summable.
Hence the above Fredholm determinant is well-defined.
As $N\to\infty$, the kernel $K_x(W,W')$ converges to $0$ for $W\in\C_\varphi\setminus\C_\varphi^\delta$ also because of Lemma~\ref{lemma:residues}.
Hence if we replace all the integrations in \eqref{Fredholmseries} on $\R$ by integrals on a fixed neighbourhood of $0$, then the error what we make goes to $0$ by dominated convergence.

On the other hand, by a similar argument, one can localize the $Z$ integrals as well.
By linearity, one can take out the $Z$ integrations from the determinant in \eqref{Fredholmseries} to obtain a sum where the $k$th term is a $2k$-fold integration.
It is still integrable, since the behaviour in the $Z$ variables is $e^{-\pi\Im(Z)}$ due to the sine in the denominator.
The function $\Re(f_0)$ is steep descent along vertical lines by Lemma~\ref{lemma:steep}, and if we choose $\sigma$ small enough,
then the steep descent property remains valid along $\D_W$ by Taylor expansion.
In a similar way as for the $W$ variable, the $Z$ integration can also be localized.
The steep descent property and the periodicity implies that by making a small error we can restrict the $Z$ integral to the set $\cup_{k\in\Z}I_k$ where
$I_k=\D_W\cap\{z:|z-(\theta+\I k2\pi/\log q)|\le\delta\}$, in particular, $I_0=\D_\varphi^\delta$.

The integral over $\cup_{k\in\Z}I_k$ reduces to the integral over $I_0$ in the limit, because if $Z\in I_k$ for $k\neq0$, then
\begin{equation}
\frac{N^{-1/3}}{\sin(\pi(W-Z))}\asymp N^{-1/3}e^{-\pi|\Im Z|}\asymp N^{-1/3}e^{-\frac{2\pi^2}{|\log q|}|k|}
\end{equation}
which is summable and gives an $\O(N^{-1/3})$ error uniformly in $W$.
For $Z\in I_0$, we are in the neighbourhood of the pole of the sine inverse, so the integral on $I_0$ is not negligible. This completes the proof.
\end{proof}

\subsection{Convergence of the kernel}

\begin{proof}[Proof of Proposition~\ref{prop:kernelconv}]
We apply the change of variables
\begin{equation}\label{localchov}
W=\theta+wN^{-1/3},\quad W'=\theta+w'N^{-1/3},\quad Z=\theta+zN^{-1/3}
\end{equation}
as in \eqref{defKxd} and use the Taylor expansions \eqref{f0Taylor}--\eqref{f2Taylor} to see that $K_{x,\delta}^N(w,w')$ gets arbitrarily close to $K_{x,\delta N^{1/3}}'(w,w')$
for any $w,w'\in\C_{\varphi,\delta N^{1/3}}$ as $N\to\infty$.

In order to get that the Fredholm determinants are also close, we need a uniform fast decaying bound on $K_{x,\delta}^N(w,w')$.
The main term in the exponent is
\begin{equation}
-N\Re f_0(W)=-\chi\Re\frac{w^3}3+\O_\epsilon({N^{-1/3}}w^4)=-\chi\Re\frac{w^3}3+\O_\epsilon(w^3)
\end{equation}
where $\O_\epsilon(w^3)$ means that the error term is at most $\epsilon w^3$.
By taking $\delta$ small enough, $\epsilon$ can be arbitrarily small, hence the error term is negligible compared to the cubic behaviour of $-\chi\Re w^3/3$.
The lower order error terms are similarly dominated.
Hence the difference of the Fredholm determinants goes to $0$ as $N\to\infty$ by dominated convergence, which proves the proposition.
\end{proof}

\begin{proof}[Proof of Proposition~\ref{prop:kernelextend}]
Since the integrand in \eqref{kernelK'} has cubic exponential decay in $w$ and $z$ along the given contours $\C_{\varphi,\infty}$ and $\D_{\varphi,\infty}$ respectively,
the convergence of the Fredholm determinants follows by dominated convergence similarly to the earlier proofs.
\end{proof}

\subsection{Reformulation of the kernel}

\begin{proof}[Proof of Proposition~\ref{prop:rewritekernel}]
By definition \eqref{kernelK'}, we can write the kernel $K_{x,\infty}'$ on $L^2(\C_{\varphi,\infty})$ as
\begin{equation}\label{kernelAB}
K_{x,\infty}'(w,w')=(AB)(w,w')
\end{equation}
where
\begin{align}
A(w,\lambda)&=e^{-\chi w^3/3-c(\log q)^2w^2/2-\beta_x w+\lambda w}\\
B(\lambda,w')&=\frac1{2\pi\I}\int_{\D_{\varphi,\infty}}\frac{\d z}{z-w'}e^{\chi z^3/3+c(\log q)^2z^2/2+\beta_x z-\lambda z}
\end{align}
with $\lambda\in\R_+$ and the composition $AB$ on the right-hand side of \eqref{kernelAB} is also meant on $L^2(\R_+)$.
The equality \eqref{kernelAB} can be seen since
\[\frac1{z-w}=\int_0^\infty\d\lambda\,e^{-\lambda(z-w)}\]
as long as $\Re(z-w)>0$.

Hence we can write
\begin{equation}
\det(\id-K_{x,\infty}')_{L^2(\C_{\varphi,\infty})}=\det(\id-AB)_{L^2(\C_{\varphi,\infty})}=\det(\id-BA)_{L^2(\R_+)}
\end{equation}
with
\begin{equation}
(BA)(a,b)=\int_0^\infty\d\lambda\frac1{(2\pi\I)^2}\int_{\C_{\varphi,\infty}}\d w\int_{\D_{\varphi,\infty}}\d z
\frac{e^{\chi z^3/3+c(\log q)^2z^2/2+(\beta_x-a-\lambda)z}}{e^{\chi w^3/3+c(\log q)^2w^2/2+(\beta_x-b-\lambda)w}}.
\end{equation}
Using the general formula
\begin{equation}
\frac1{2\pi\I}\int_{\D_{\varphi,\infty}}\exp\left(a\frac{z^3}3+bz^2+cz\right)\d z=a^{-1/3}\exp\left(\frac{2b^3}{3a^2}-\frac{bc}a\right)\Ai\left(\frac{b^2}{a^{4/3}}-\frac c{a^{1/3}}\right),
\end{equation}
we get that
\begin{equation}
(BA)(a,b)=\chi^{-1/3}e^{c(\log q)^2(a-b)/(2\chi)}\int_0^\infty\d\lambda\Ai\left(\frac a{\chi^{1/3}}+x+\lambda\right)\Ai\left(\frac b{\chi^{1/3}}+x+\lambda\right)
\end{equation}
by \eqref{defbeta} and after a change of variable.
After conjugation by the exponential prefactor and by rescaling the Fredholm determinant, we get that
\begin{equation}
\det(\id-BA)_{L^2(\R_+)}=\det(\id-K_{\Ai,x})_{L^2(\R_+)}=F_{\rm GUE}(x)
\end{equation}
as required.
\end{proof}

With the proof of Proposition~\ref{prop:rewritekernel}, we also conclude Theorem~\ref{thm:Fredholmconv} which is the key for the main result of the present paper.
We add a final remark on an alternative way of proving Theorem~\ref{thm:Fredholmconv}.

\begin{remark}
There is a slightly different way of proving Theorem~\ref{thm:Fredholmconv}.
Instead of choosing $\varphi$ to be very close to $\pi/2$, it is also possible to take $\varphi=\pi/2$.
This alternative requires about the same complexity of the proof as the way followed here and it yields the following differences.
The limit of the contour $\C_\varphi$ as one rescales the neighbourhood of $\theta$ is $\I\R$,
so the contour $\C_\varphi$ has to be locally modified to another steep descent contour before the localization of the integral in Proposition~\ref{prop:localization}.
This needs more argument.
The decay of the kernel $K_x$ along $\C_{\pi/2}$ in Lemma~\ref{lemma:expdecay} is not exponential but polynomial in $s$,
but the power is proportional to $N$ which is still enough decay for the rest of the proof.
The argument in the proof of Lemma~\ref{lemma:residues} is slightly simpler.
\end{remark}

\section{Confirmation of KPZ scaling theory for \mbox{$q$-TASEP}}\label{s:scalingtheory}

This section contains the argument which shows that the KPZ scaling theory conjecture formulated in Conjecture~\ref{conj:KPZ} is satisfied for \mbox{$q$-TASEP}.
\begin{proof}[Proof of Theorem~\ref{thm:scalingtheory}]
We start with \eqref{currentrandomin} for $c=0$ which reads
\begin{equation}\label{currentrandomin2}
H\left(\tau,\frac{f-1}\kappa \tau+\frac{\chi^{1/3}}{\kappa^{1/3}\log q}\xi_\tau \tau^{1/3}\right)=\frac\tau\kappa
\end{equation}
where $\xi_\tau$ has asymptotically Tracy--Widom distribution.
In \eqref{currentrandomin2}, $(f-1)/\kappa$ and $1/\kappa$, the coefficients of $\tau$ are functions of $\theta$.
By modifying $\theta$ on the $\tau^{-2/3}$ scale, we can choose $\theta'$ such that the corresponding $(f-1)/\kappa$ is equal to the second argument of $H$ in \eqref{currentrandomin2}.
Then the right-hand side of \eqref{currentrandomin2} changes by Taylor expansion.
By using the series representation of $\kappa$ and $f$ given in \eqref{kappafseries}, one can see that
\begin{equation}\label{implicitderivative}
\frac\d{\d\theta}\frac{f-1}\kappa\left(\frac\d{\d\theta}\frac1\kappa\right)^{-1}=-(q^\theta\kappa-f+1).
\end{equation}
The first part of the theorem now follows by Taylor expansion.

Next we show that the scaling KPZ conjecture is satisfied.
The \mbox{$q$-TASEP} in~\cite{Spo13a} is defined such that particles are initially on the non-negative integer positions and they all try to jump to the left with the corresponding rates.
Therefore, the particle current $H$ defined in Remark~\ref{rem:current} and the height function $h$ used in~\cite{Spo13a} are related via $h(y,t)=2H(t,-y-1)+y$ (cf.~\eqref{hHrelation}), hence \eqref{currentrandomout} reads
\begin{equation}\label{currentSpohn}\begin{aligned}
h\left(-\frac{f-1}\kappa\tau,\tau\right)&=2H\left(\tau,-\left(-\frac{f-1}\kappa\right)\tau-1\right)+\frac{f-1}\kappa\tau\\
&=\left(\frac2\kappa+\frac{f-1}\kappa\right)\tau+\frac2{q^\theta\kappa-f+1}\frac{\chi^{1/3}}{\kappa^{1/3}\log q}\xi_\tau\tau^{1/3}.
\end{aligned}\end{equation}
The profile function is
\begin{equation}
\phi:-\frac{f-1}\kappa\mapsto\frac2\kappa+\frac{f-1}\kappa.
\end{equation}
Its derivative determines $\varrho$ in~\cite{Spo13a} that is an affine transform of the particle density given in Proposition~\ref{prop:lln}.
One has
\begin{equation}\label{rho1}
1+\varrho=1+\phi'=\frac2{q^\theta\kappa-f+1}=\frac{2\log q}{\log q+\log(1-q)+\Psi_q(\theta)}
\end{equation}
using \eqref{implicitderivative} in the second equality and \eqref{defkappa}--\eqref{deff} in the last one.
On the other hand, (2.21) in~\cite{Spo13a} yields
\begin{equation}\label{rho2}
1+\varrho=\frac2{1-\alpha\frac\d{\d\alpha}\log(\alpha;q)_\infty}=\frac2{1+\sum_{m=0}^\infty\frac{\alpha q^m}{1-\alpha q^m}}=\frac{2\log q}{\log q+\log(1-q)+\Psi_q(\log_q\alpha)}
\end{equation}
which means that the parameter of the stationary measure is determined via
\begin{equation}\label{alphatheta}
\alpha=q^\theta.
\end{equation}
This choice of parameters is in accordance with Proposition~\ref{prop:lln} and with the Legendre transform representation of $\phi$ given also in (2.10) of~\cite{Spo13a}.

Formulas (2.22)--(2.23) in~\cite{Spo13a} give that
\begin{equation}\label{jAqtasep}
j(\varrho)=-\alpha(1+\varrho),\qquad A(\varrho)=-\alpha(1+\varrho)\frac{\d\varrho}{\d\alpha}
\end{equation}
for \mbox{$q$-TASEP}.
By computing the first two derivatives of the inverse function $\alpha\mapsto\varrho$ given in \eqref{rho2}, one can see that
\begin{equation}\label{lambdacalc}
\lambda=-\frac{\d^2}{\d\varrho^2}j(\varrho)=\frac{q^\theta(\Psi_q'(\theta)\log q-\Psi_q''(\theta))(\log q+\log(1-q)+\Psi_q(\theta))^3}{2(\Psi_q'(\theta))^3}
\end{equation}
with the notation \eqref{alphatheta}.
By similar calculations for the second quantity in \eqref{jAqtasep}, we have
\begin{equation}\label{Acalc}
A=\frac{4\Psi_q'(\theta)\log q}{(\log q+\log(1-q)+\Psi_q(\theta))^3}.
\end{equation}
Putting \eqref{lambdacalc} and \eqref{Acalc} together, the coefficient of $st^{1/3}$ on the right-hand side of the inequality sign in \eqref{scalingtheory} is given by
\begin{equation}\label{coeff}
-(-{\textstyle\frac12}\lambda A^2)^{1/3}=\frac{2^{2/3}q^{\theta/3}((\log q)^2)^{1/3}(\Psi_q'(\theta)\log q-\Psi_q''(\theta))^{1/3}}{(\log q+\log(1-q)+\Psi_q(\theta))(\Psi_q'(\theta))^{1/3}}
\end{equation}
where $((\log q)^2)^{1/3}$ is understood as a positive number.
Note that this is the same as the coefficient of the $\xi_\tau\tau^{1/3}$ random term on the right-hand side of \eqref{currentSpohn}, which completes the proof since the coefficient in \eqref{coeff} is negative.
\end{proof}

\begin{proof}[Proof of Corollary~\ref{cor:density}]
By Remark~\ref{rem:techcond}, Theorem~\ref{thm:scalingtheory} certainly holds and confirms the law of large numbers \eqref{lln} for all $\theta$ for which $(f-1)/\kappa\in(-(1-q),0)$.
It yields that the value of the particle density and the current given in \eqref{rhojrho} follows by differentiation \eqref{implicitderivative} and by using \eqref{kappafseries}.
\end{proof}

\section*{Acknowledgements}
The authors thank Ivan Corwin for his valuable comments on the earlier version of the paper and for drawing their attention to the paper \cite{Spo13a}.
The work of P.L.\ Ferrari is supported by the German Research Foundation via the SFB 1060--B04 project.
B.\ Vet\H o is grateful for the generous support of the Humboldt Research Fellowship for Postdoctoral Researchers during his stay at the University
of Bonn and for the Postdoctoral Fellowship of the Hungarian Academy of Sciences.
His work is partially supported by OTKA (Hungarian National Research Fund) grant K100473.


\begin{thebibliography}{10}

\bibitem{B14}
G.~Barraquand, \emph{A phase transition for $q$-{TASEP} with a few slower particles} (2014), arXiv:1404.7409.

\bibitem{BC13}
A.~Borodin and I.~Corwin, \emph{Discrete time $q$-{TASEP}s}, to appear in Int. Math. Res. Not. IMRN (2014), arXiv:1305.2972.

\bibitem{BC11}
A.~Borodin and I.~Corwin, \emph{Macdonald processes}, Probab. Theory Relat. Fields {\bf 158} (2014), 225--400.

\bibitem{BCF12}
A.~Borodin, I.~Corwin, and P.L. Ferrari, \emph{Free energy fluctuations for directed polymers in random media in $1+1$ dimension}, Comm. Pure Appl. Math. \textbf{67} (2014), 1129--1214.

\bibitem{BCPS13}
A.~Borodin, I.~Corwin, L.~Petrov, and T.~Sasamoto, \emph{Spectral theory for the q-Boson particle system}, to appear in Compositio Mathematicae (2014), arXiv:1308.3475.

\bibitem{BCS12}
A.~Borodin, I.~Corwin, and T.~Sasamoto, \emph{From duality to determinants for $q$-{TASEP} and {ASEP}}, to appear in Ann. Probab. (2014), arXiv:1207.5035.

\bibitem{BF07}
A.~Borodin and P.L. Ferrari, \emph{Large time asymptotics of growth models on space-like paths I: PushASEP}, Electron. J. Probab. \textbf{13} (2008), 1380--1418.

\bibitem{BFPS06}
A.~Borodin, P.L. Ferrari, M.~Pr{\"a}hofer, and T.~Sasamoto, \emph{Fluctuation properties of the TASEP with periodic initial configuration}, J. Stat. Phys. \textbf{129} (2007), 1055--1080.

\bibitem{CFP10b}
I.~Corwin, P.L. Ferrari, and S.~P{\'e}ch{\'e}, \emph{Universality of slow decorrelation in KPZ models}, Ann. Inst. H. Poincar\'e Probab. Statist. \textbf{48} (2012), 134--150.

\bibitem{CP13}
I.~Corwin and L.~Petrov, \emph{The $q$-{P}ush{ASEP}: A new integrable model for traffic in $1+1$ dimension}, (2013), arXiv:1305.2972.

\bibitem{Fer10b}
P.L. Ferrari, \emph{From interacting particle systems to random matrices}, J. Stat. Mech. (2010), P10016.

\bibitem{SI07}
T.~Imamura and T.~Sasamoto, \emph{Dynamical properties of a tagged particle in the totally asymmetric simple exclusion process with the step-type initial condition},
J. Stat. Phys. \textbf{128} (2007), 799--846.

\bibitem{Jo03b}
K.~Johansson, \emph{Discrete polynuclear growth and determinantal processes}, Comm. Math. Phys. \textbf{242} (2003), 277--329.

\bibitem{KL13}
M.~Korhonen and E.~Lee, \emph{The transition probability and the probability for the left-most particle's position of the $q$-{TAZRP}} (2013), arXiv:1308.4769.

\bibitem{OCon09}
N.~O'Connell, \emph{Directed polymers and the quantum Toda lattice}, Ann.  Probab. \textbf{40} (2012), 437--458.

\bibitem{Sas05}
T.~Sasamoto, \emph{Spatial correlations of the {1D KPZ} surface on a flat substrate}, J. Phys. A \textbf{38} (2005), L549--L556.

\bibitem{SW98b}
T.~Sasamoto and M.~Wadati, \emph{Exact results for one-dimensional totally asymmetric diffusion models}, J. Phys. A: Math. Gen. \textbf{31} (1998), 6057--6071.

\bibitem{Spo13a}
H.~Spohn, \emph{KPZ scaling theory and the semi-discrete directed polymer model}, MSRI Proceedings (2012) arXiv:1201.0645.

\bibitem{TW94}
C.A.~Tracy and H.~Widom, \emph{Level-spacing distributions and the Airy kernel}, Comm. Math. Phys. \textbf{159} (1994), 151--174.

\end{thebibliography}
\end{document}